\documentclass[12pt,psamsfonts]{amsart}
\usepackage{amssymb,eucal}
\usepackage[applemac]{inputenc}
\usepackage[T1]{fontenc}
\usepackage{yfonts}
\usepackage{graphicx}
\usepackage{epstopdf}
\usepackage{enumerate}

\usepackage{mathtools}

\def\Z{\mathbb Z}

\theoremstyle{plain}
\newtheorem{theorem}{Theorem}
\newtheorem{proposition}{Proposition}

\theoremstyle{definition}

\theoremstyle{remark}
\newtheorem{remark}{Remark}

\begin{document}

\title{The Quartic Residues Latin Square}
\author{Christian Aebi and Grant Cairns}
\address{Coll\`ege Calvin, Geneva, Switzerland 1211}
\email{christian.aebi@edu.ge.ch}
\address{La Trobe University, Melbourne, Australia 3086}
\email{G.Cairns@latrobe.edu.au}

\begin{abstract}
We establish an elementary, but rather striking  pattern concerning the quartic residues of primes $p$ that are congruent to 5 modulo 8.   Let $g$ be a generator of the multiplicative group of $\Z_p$ and let $M$ be the $4\times 4$ matrix whose $(i+1),(j+1)-$th entry is the number of elements $x$ of $\Z_p$ of the form $x\equiv g^k \pmod p$ where $k\equiv i \pmod 4$ and $\lfloor 4x/p \rfloor = j$, for $i,j=0,1,2,3$. We show that $M$ is a Latin square, provided the entries in the first row are distinct, and that $M$  is essentially independent of the choice of $g$. As an application, we prove that the  sum in $\Z$ of the quartic residues is
$
\frac{p}5(M_{11}+2M_{12}+3M_{13}+4M_{14})$.
\end{abstract}

\maketitle

\section{Introduction}
Let $p$ be prime with $p\equiv 5\pmod 8$,  identify the finite field $\Z_p$ with the set $\{0,1,\dots,p-1\}$, and let $g$ be an arbitrary generator of the multiplicative group $\Z_p^*$.  For $i=0,1,2,3$, let $B_i$ be the set of elements $x$ of $\Z_p^*$ of the form
$x\equiv g^k \pmod p$, where $k\equiv i\pmod 4$. Note that although $k$ is not unique,   its value modulo 4 is unique, by Fermat's little theorem. Now divide $\Z_p^*$ into 4 consecutive subintervals; for $j=0,\dots,3$, let 
\[
I_j=\left\{x\in \Z_p^*: \frac{jp}4 <x< \frac{( j+1)p}4\right\}.
\]
Consider the $4\times 4$ matrix $M$ whose $(i+1),(j+1)-$th entry  is the number of elements in the intersection $B_{i}\cap I_j$. 
Note that the first row counts the quartic residues in the intervals $I_0,I_1,I_2,I_3$.
For example, consider the prime $p=109$ and the generator $g=  6$. The quartic residues are:
\[
1,3,5,7,9,15,16,21,22,25,26,27\ | | \ 35,38,45,48,49\ | | \ 63,66,73,75,78,80,81\ | | \ 89,97,105.\]
One finds that $M$ is a Latin square:
 \[
M=\begin{bmatrix}  12&5&7&3\\
7&12&3&5\\
3&7&5&12\\
5&3&12&7
\end{bmatrix}.  
\]
This pattern is  surprising, as the sets $B_i$ are given by the multiplicative structure, while the sets $I_j$ are related to the additive structure. So the definition of $M$  is rather peculiar, if not downright perverse. Nevertheless, we will show that the pattern given above always occurs, provided  $p\equiv 5\pmod 8$, and $M$ is essentially independent of the choice of generator $g$.

 Recall that in $\Z_p^*$ the number of generators is given by the Euler function $\varphi(p-1)$. In particular, there is an even number of generators. In fact, half of the generators $g$ have the property $2\equiv g^k \pmod p$ where $k\equiv1 \pmod 4$, and half have $2\equiv g^k \pmod p$ with $k\equiv3 \pmod 4$. To see this, notice that $2$ is not a quadratic residue  as $p\equiv 5\pmod 8$, so if $2\equiv g^k \pmod p$, then $k$ is necessarily odd.  Furthermore, if $g$ is a generator, then so too is $g^{-1}$, and $g^k \equiv (g^{-1})^{-k}\equiv (g^{-1})^{p-1-k} \pmod p$, by Fermat's little theorem. And as $p\equiv1 \pmod 4$, one has $k\equiv1 \pmod 4$ if and only if $p-1-k\equiv3 \pmod 4$.

\begin{theorem}\label{squares} Suppose that  $p$ is prime with $p\equiv 5\pmod 8$ and choose a generator $g$ of the multiplicative group $\Z_p^*$. Then the matrix $M$ has one the following two forms:
\begin{equation}\label{M}
\begin{bmatrix}  a&b&c&d\\
c&a&d&b\\
d&c&b&a\\
b&d&a&c
\end{bmatrix}\qquad\text{or}\qquad  
\begin{bmatrix}  a&b&c&d\\
b&d&a&c\\
d&c&b&a\\
c&a&d&b
\end{bmatrix}, 
\end{equation}
depending on whether $2\equiv g^k \pmod p$ with $k\equiv1 \pmod 4$ or $k\equiv3 \pmod 4$ respectively. In particular, in both cases, if the entries in the first row are distinct, the matrix $M$ is a Latin square.
\end{theorem}

\begin{remark} The first row of the matrix $M$ is independent of the generator $g$, as it depends only on the quartic residues. So the theorem shows that $M$ is independent of the choice of generator, up to a transposition of rows 2 and 4.
\end{remark}

The paper is organised as follows. In Section \ref{pf} we prove Theorem \ref{squares} and we  show how $M$ may be defined in a different manner.
In Section \ref{ss}, as an application of  Theorem \ref{squares}, we give a formula for the sum of the quartic residues, in terms of $p$ and the entries in the first row of $M$.

\section{Proof of Theorem  \ref{squares}}\label{pf}

\begin{proof}[Proof of Theorem  \ref{squares}]
Let $g$ be a generator of  $\Z_p^*$. First note that writing $p=8n+5$, one has $-1\equiv  g^{4n+2}$ in $\Z_p^*$. So modulo $p$, multiplication by $-1$ sends $B_i$ to $B_{i+2}$, for $i=0,1$. Hence, considering  the map $x\mapsto p-x$, we see that  the first row of $M$ equals the third row in reverse order, and visa versa. Similarly, the second row equals the fourth row in reverse order. To enable us to refer to this property below, let us record it formally: 
\begin{equation}\label{neg}
M_{ij}=M_{(i+2),(5-j)},
\end{equation}
where $1\leq i\leq 2$ and $1\leq j\leq 4$. So it remains to consider how row 2 is obtained from row 1.
Let us suppose that $2\equiv g^k \pmod p$ with $k\equiv1 \pmod 4$, so  modulo $p$, multiplication by 2  sends $B_i$ to $B_{i+1}$ for $i=0,1,2$.  (The case $k\equiv3 \pmod 4$ is entirely analogous; multiplication by 2  sends $B_i$ to $B_{i-1}$). We claim that 
\begin{equation}\label{rows12}
M_{21}=M_{13},\quad M_{22}=M_{11}.
\end{equation}
Let us see how \eqref{rows12} gives the theorem. To lighten the notation, let $M_{11}=a,M_{12}=b,M_{13}=c,M_{14}=d$.
By \eqref{neg} and \eqref{rows12}, $M$ has the following form:
\[
M=\begin{bmatrix}  a&b&c&d\\
c&a&*&*\\
d&c&b&a\\
*&*&a&c
\end{bmatrix}.  
\]
Now notice that the  rows and columns of $M$ all have the same sum, $(p-1)/4$. Comparing the first row with the four columns we obtain
the form of the left hand side of \eqref{M}.

It remains  to establish \eqref{rows12}. For this, we require some further notation. We will look at the bottom and top halves of the intervals $I_{j-1}$, and then  the various quarters of the $I_{j-1}$. For each $i,j,k=1,\dots,4$,
let 
\begin{align*}
b_{ij}&=|B_{i-1} \cap  \{x\in \Z_p^*: (j-1)p/4 < x<  (2j-1)p/8\}|,\\
t_{ij}&=|B_{i-1} \cap\{x\in \Z_p^*: (2j-1)p/8 < x<  jp/4\}|,\\
q_{jk}&=|B_{0} \cap\{x\in \Z_p^*: (4j-5+k)p/16 < x<  (4j-4+k)p/16\}|,
\end{align*}
so that by definition,
\begin{equation}\label{halves}
M_{ij}=b_{ij}+t_{ij}
\end{equation}
and 
\begin{equation}\label{quarters}
b_{1j}=q_{j1}+q_{j2}, \qquad t_{1j}=q_{j3}+q_{j4}.
\end{equation}
 
Considering the effect modulo $p$ of multiplication of the elements of $B_{i-1}$ by $2$, we have
for $i=1,2$ and $3$,
\begin{equation}\label{3s}
M_{(i+1)1} =b_{i1}+ b_{i3},\quad 
M_{(i+1)2} =t_{i1}+ t_{i3},\quad
M_{(i+1)3} =b_{i2}+ b_{i4},\quad
M_{(i+1)4} =t_{i2}+ t_{i4}.
\end{equation}
Similarly,
\begin{equation}\label{bs}
b_{22} =q_{13}+ q_{33},\qquad
b_{24} =q_{23}+ q_{43},
\end{equation}
and
\begin{equation}\label{ts}
t_{22} =q_{14}+ q_{34},\qquad
t_{24} =q_{24}+ q_{44}.
\end{equation}
Thus, \eqref{neg} with $i=1$, $j=1,2$, \eqref{3s} with $i=2$, \eqref{bs} and \eqref{ts}  give
\begin{align*}
M_{11}=M_{34}=t_{22}+ t_{24}= q_{14}+ q_{24}+ q_{34}+ q_{44},\\
M_{12}=M_{33}=b_{22}+ b_{24}= q_{13}+ q_{23}+ q_{33}+ q_{43}.
\end{align*}
Adding the above two equations and using \eqref{quarters}, we get
$M_{11}+M_{12}= t_{11}+t_{12}+t_{13}+t_{14}$,
and thus  by \eqref{halves},
\begin{equation}\label{plus}
b_{11}+b_{12}= t_{13}+t_{14}.
\end{equation}
As the first row of $M$ has the same sum as the first column of $M$, we have
\[
M_{12}+M_{13}+M_{14}=M_{21}+M_{31}+M_{41}.
\]
But by \eqref{neg}, we have $M_{31}=M_{14}$ and $M_{41}=M_{24}$. Hence, using \eqref{3s} with $i=1$,
we have $M_{12}+M_{13}=M_{21}+M_{24}=b_{11}+b_{13}+t_{12}+t_{14}$,
and hence by \eqref{halves},
\begin{equation}\label{minus}
b_{12}+t_{13}=b_{11}+t_{14},
\end{equation}
Solving \eqref{plus} and \eqref{minus} we obtain
$b_{11}=t_{13},b_{12}=t_{14}$.
Finally, by \eqref{halves} and \eqref{3s} with $i=1$,
\begin{align*}
M_{21}&=b_{11}+b_{13}=t_{13}+t_{13}= M_{13}\\
M_{22}&=t_{11}+t_{13}=b_{11}+t_{11}= M_{11}.
\end{align*}
which gives \eqref{rows12} as required.
\end{proof}

\begin{remark} As we saw in the proof of Theorem \ref{squares}, if $x$ is a quartic residue, so too is $-4x$, modulo $p$, and similarly, so too is $-x/4$. In particular,  $1$ is a quartic residue since $p\equiv 1\pmod 4$, and hence $p-4$ and  $(p-1)/4$ are quartic residues too.  
\end{remark}

We now show that the matrix $M$ can be defined in a different manner, by replacing the sets $I_0,\dots,I_3$ by a different partition of $\Z_p^*$ into 4 subsets of equal cardinality. For $j=1,\dots,4$, let $J_j=\{x\in \Z_p^* : x\equiv j \pmod{4}\}$. We will show that the distribution of the quartic residues amongst the intervals $I_0,\dots,I_3$ is identical to the distribution of the quartic residues amongst the conjugacy classes $J_1,\dots,J_4$.  

\begin{proposition} \label{mod} For all $i,j=1,\dots,4$, one has $M_{ij}= |B_{i-1} \cap J_j |$.
\end{proposition}

For example, consider $p=109$. The quartic residues, collected in the intervals $I_0,\dots,I_3$ are:
\[
1,3,5,7,9,15,16,21,22,25,26,27\ | | \ 35,38,45,48,49\ | | \ 63,66,73,75,78,80,81\ | | \ 89,97,105,\]
while collected in the intervals $J_1,\dots,J_4$, they are:
\[
1,5,7,21,25,45,49,73,81,89,97,105\ | | \ 22,26,38,66,78\ | | \ 3,7,15,27,35,63,75\ | | \ 16,48,80.\]
The collections are different, but as the Proposition asserts, the cardinalities are the same.

\begin{proof}[Proof of Proposition \ref{mod}]  We only treat the first row of $M$; the other rows are dealt with in exactly the same manner. First suppose that the generator $g$ is such that $2\equiv g^k \pmod p$ where $k\equiv1 \pmod 4$. We will show that working modulo $p$, the map $x\mapsto -4x$ establishes the required bijections.

Recall that $M_{11}=| B_0 \cap I_0 |$. Multiplying the  set $B_0 \cap I_0$ by 2 gives the even elements of $B_1\cap(I_0\cup I_1)$. Multiplying  again by 2, we obtain the elements of $B_2$ that are divisible by 4. Finally, multiplying the latter set by $-1$ and adding $p \equiv 1 \pmod{4}$ to the products we obtain the elements of $B_0$ congruent to $1\pmod 4$, as asserted by the proposition. Symbolically, we can represent the above bijective transformations in the following way:
\[
 B_0 \cap I_0\   \xrightarrow[\quad\cong\quad]{ \quad 2x \quad}\  B_1 \cap (I_0 \cup I_1)\cap (J_2\cup J_4)\   \xrightarrow[\quad\cong\quad]{ \quad 2x \quad}\  B_2\cap J_4\ \xrightarrow[\quad\cong\quad]{ \quad -x+p \quad}\  B_0\cap J_1.
\]
Using the above convention, we obtain concerning $M_{12}$ :
\[
 B_0 \cap I_1\   \xrightarrow[\quad\cong\quad]{ \quad 2x \quad}\  B_1 \cap (I_2 \cup I_3)\cap (J_2\cup J_4)\   \xrightarrow[\quad\cong\quad]{ \quad 2x-p \quad}\  B_2\cap J_3\ \xrightarrow[\quad\cong\quad]{ \quad -x+p \quad}\  B_0\cap J_2.\]
Similarly concerning $M_{13}$ and then $M_{14}$ we get:
\[
 B_0 \cap I_2\   \xrightarrow[\quad\cong\quad]{ \quad 2x-p \quad}\   B_1 \cap (I_0 \cup I_1)\cap (J_1\cup J_3)\  \xrightarrow[\quad\cong\quad]{ \quad 2x \quad}\  B_2\cap J_2 \ \xrightarrow[\quad\cong\quad]{ \quad -x+p \quad}\  B_0\cap J_3,\]
\[
 B_0 \cap I_3\  \xrightarrow[\quad\cong\quad]{ \quad 2x-p \quad}\   B_1 \cap (I_2 \cup I_3)\cap (J_1\cup J_3)\  \xrightarrow[\quad\cong\quad]{ \quad 2x-p \quad}\  B_2\cap J_1\ \xrightarrow[\quad\cong\quad]{ \quad -x+p \quad}\  B_0\cap J_4.
\]
This establishes the proposition when $2\equiv g^k \pmod p$ with $k\equiv1 \pmod 4$.
 The case where  $k\equiv3 \pmod 4$ is treated in exactly the same manner, with $B_1$ replaced by $B_3$.
\end{proof}

\section{The sum of the quartic residues}\label{ss}

In this section, as an application of Theorem \ref{squares}, we show that the sum of the quartic residues can be computed from the entries in the first row of $M$. To place this result in context, let us first recall that some well known facts about sums of quartic residues.

Suppose that $p\equiv 3 \pmod 4$. Then every quadratic residue is a quartic residue. Indeed,  the map $x\mapsto p-x$ interchanges the quadratic residues with the nonquadratic residues. So if in $\Z_p$ we have $a\equiv b^2 $, then either $b$ is a quadratic residue, in which case $a$ is quartic residue, or $b$ is not a quadratic residue, in which case $p-b$ is a quadratic residue and $a\equiv (p-b)^2$, and once again $a$ is a quartic residue. 
In particular, the sum of the quartic residues equals the sum of the quadratic residues, and we can use V.-A.~Lebesgue's results  \cite{AC}.

 Suppose that $p\equiv 1 \pmod 8$. By Gauss, $-1$ is a quartic residue. So the map $x\mapsto p-x$ fixes the quartic residues. Thus, for each quartic residue $a$, the element $p-a$ is also a quartic residue and the sum of this pair is therefore $p$. So the sum of the quartic residues is $p/2$ times the number $(p-1)/4$ of the quartic residues. Thus
the sum of the quartic residues is $p(p-1)/4$.

The above comments show that the interesting case is when $p\equiv 5\pmod 8$. 

\begin{theorem}\label{sums}  For all primes $p$ with $p\equiv 5\pmod 8$, the sum in $\Z$ of the quartic residues is
\[
\frac{p}5(M_{11}+2M_{12}+3M_{13}+4M_{14}).
\]
\end{theorem}

\begin{proof} Choose  a generator $g$ of  $\Z_p^*$ such that $2\equiv g^k \pmod p$  where $k\equiv1 \pmod 4$, and let $M$ be the resulting matrix, as in Theorem \ref{squares}.  Let $b_i$ denote the sum of the elements of $B_i$, for $i=0,\dots,3$.  Consider multiplication of the elements of $B_{0}$ by 2. Modulo $p$, the set $B_{0}$ is sent bijectively to $B_1$. In fact, the elements of $B_{0}\cap (I_1\cup I_2)$ are sent directly to elements of $B_{1}$, while for each element $x\in B_{0}\cap (I_3\cup I_4)$, one has $2x-p \in B_{1}$. Since there are $M_{13}+M_{14}$ elements in $B_{0}\cap (I_3\cup I_4)$, we have
\[ b_{1}= 2b_0-p(M_{13}+M_{14}).
\]
Similarly,
\[ b_{2}= 2b_1-p(M_{23}+M_{24}),\quad b_{3}= 2b_2-p(M_{33}+M_{34})
,\quad  b_{0}= 2b_3-p(M_{43}+M_{44}).
\]
Rearranging, and applying  Theorem 1, we have
\begin{align*}
2b_0&= b_1+p(M_{13}+M_{14})\\
2b_1&= b_2+p(M_{14}+M_{12})\\
2b_2&= b_3+p(M_{12}+M_{11})\\
2b_3&= b_0+p(M_{11}+M_{13}).
\end{align*}
Solving this system of 4 equations for $b_0$ gives the required result.
\end{proof}

\providecommand{\bysame}{\leavevmode\hbox to3em{\hrulefill}\thinspace}
\providecommand{\MR}{\relax\ifhmode\unskip\space\fi MR }

\providecommand{\MRhref}[2]{%
  \href{http://www.ams.org/mathscinet-getitem?mr=#1}{#2}
}
\providecommand{\href}[2]{#2}

\end{document}